\documentclass[11pt]{article}
\normalsize
\usepackage{amsfonts}
\usepackage{amsmath}
\usepackage{amssymb}

\newtheorem{thm}{Theorem}
\newtheorem{prop}[thm]{Proposition}

\newtheorem{definition}{Definition}
\newtheorem{cor}[thm]{Corollary}

\newtheorem{remark}{Remark}

\newenvironment{proof}{\noindent {\em Proof.}}{\hspace*{\fill} $\Box $\newline}

\newcommand{\wt}{\mbox{\rm wt}}

\def\ds{\displaystyle}

\title{ Singly-even self-dual codes with minimal shadow}
\author{Stefka Bouyuklieva\footnote{Supported by the
Humboldt Foundation. On leave from Faculty of Mathematics and Informatics, Veliko Tarnovo University, 5000 Veliko Tarnovo, Bulgaria}
\hspace{2mm} and
Wolfgang Willems\\[1ex]
Faculty of Mathematics,
University of Magdeburg \\ 39016 Magdeburg, Germany\\
}

\date{}

\begin{document}
\maketitle

\begin{abstract}
In this note we investigate extremal singly-even self-dual codes with minimal shadow.
For particular parameters we prove non-existence of such codes. By a result of Rains \cite{Rains-asymptotic},
the length of extremal singly-even self-dual codes is bounded. We give explicit bounds in case the shadow is minimal.
\end{abstract}


\section{  Introduction}


Let $C$ be a singly-even self-dual $[n, \frac{n}{2},d]$ code and let $C_0$ be its
doubly-even subcode. There are three cosets $C_1,C_2,C_3$ of $C_0$
such that $C_0^\perp = C_0 \cup C_1 \cup C_2 \cup C_3 $, where $C =
C_0 \cup C_2$. The set $S = C_1 \cup C_3=C_0^\perp \setminus C$ is
called the shadow of $C$. Shadows for self-dual codes were
introduced by Conway and Sloane \cite{CS} in order to derive new
upper bounds for the minimum weight of singly-even self-dual codes
and to provide restrictions on their weight enumerators.

According to \cite{Rains}
the minimum weight $d$ of a self-dual code of length $n$ is
bounded by $4[n/24]+4$ for $n \not\equiv 22 \pmod{24}$ and by $4[n/24]+6$  if $ n \equiv 22 \pmod {24}$. We call a
self-dual code meeting this bound extremal. Note that for
some lengths, for instance length $34$, no extremal self-dual codes
exist.

Some
properties of the weight enumerator of $S$ are given in the
following theorem.

\begin{thm}{\rm \cite{CS}}
Let $S(y) = \sum_{r=0}^n B_ry^r$ be the weight enumerator of $S$. Then
\begin{itemize}
\item $B_r = B_{n - r}$ for all $r$,
\item $B_r = 0$ unless $r\equiv n /2\pmod 4$,
\item $B_0 = 0$,
\item $B_r\le 1$ for $r < d /2$,
\item $B_{d /2}\le 2n / d$,
\item at most one $B_r$ is nonzero for $r < (d + 4 )/2$.
\end{itemize}
\end{thm}

Elkies studied in \cite{Elkies} the minimum weight $s$
(respectively the minimum norm) of the shadow of self-dual codes
(respectively of unimodular lattices), especially in the cases
where it attains a high value. Bachoc and Gaborit proposed to
study the parameters $d$ and $s$ simultaneously
\cite{Bachoc_Gaborit}. They proved that $2d+s\le \frac{n}{2}+4$,
except in the case $n\equiv 22\pmod{24}$ where $2d+s\le
\frac{n}{2}+8$. They called the codes attaining this bound
\textit{$s$-extremal}. In this note we study singly-even self-dual
codes for which the minimum weight of the shadow has smallest possible value.
possible.

\begin{definition}
We say that a self-dual code $C$ of length $24m+8l+2r$  with $r=1,2,3$ and $l=0,1,2$ is a code with minimal shadow if $\wt(S)=r$. For $r=0$, $C$  is called of minimal shadow if $\wt(S)=4$.
\end{definition}

Self-dual codes with minimal shadow are subject of two previous
articles.  The paper
\cite{StefkaZlatko} is devoted to connections between 
self-dual codes of length $24m+8l+2$ with $\wt(S)=1$,
combinatorial designs and secret sharing schemes. The structure of
these codes are used to characterize access groups in a
secret sharing scheme based on codes. There are two types of schemes which are
 proposed -  with one-part secret and with two-part secret.
Moreover, some of the considered codes support 1- and 2-designs.
The performance of the extremal self-dual codes of length $24m+8l$
where $l=1,2$ have been studied in \cite{performance}.  In
particular, different types of codes with the same parameters are
 compared with regard to the decoding error probability. It
turned out that for lengths $24m+8$ singly-even codes with minimal
shadow perform better than doubly-even codes. Thus from the point
of view of data correction one is interested in singly-even codes
with minimal shadow. \\

This article is organized as follows. In Section 2 we prove that extremal self-dual codes with minimal shadow of length $24m+2t$ for $t=1$, 2, 3, 5, 11 do not exist. Moreover, for $t=4$, 6, 7 and 9, we obtain upper bounds for the length.
We also prove  that if extremal doubly-even self-dual codes of length $n=24m+8$ or $24m+16$ do not exist then extremal singly-even self-dual codes with minimal shadow do not exist for the same length. The only case for which we do not have a bound for the length is $n=24m+20$.

All computations have been carried out with Maple.

\section{Extremal self-dual codes with minimal shadow}

Let $C$ be a singly-even self-dual code of length $n=24m+8l+2r$ where $l=0,1,2$ and $r=0,1,2,3$.  The weight enumerator of $C$ and its
shadow are given by \cite{CS}:
$$W(y)=\sum_{j=0}^{12m+4l+r}a_jy^{2j}=\sum_{i=0}^{3m+l}c_i(1+y^2)^{12m+4l+r-4i}(y^2(1-y^2)^2)^i$$
$$S(y)=\sum_{j=0}^{6m+2l}b_jy^{4j+r}=\sum_{i=0}^{3m+l}(-1)^ic_i2^{12m+4l+r-6i}y^{12m+4l+r-4i}(1-y^4)^{2i}$$

Using these expressions we can write $c_i$ as a linear
combination of the $a_j$ and as a linear combination of the $b_j$
in the following way \cite{Rains}:
\begin{equation} \label{eq1}
c_i=\sum_{j=0}^i \alpha_{ij}a_j=\sum_{j=0}^{3m+l-i}\beta_{ij}b_j.
\end{equation}

Suppose $C$ is an extremal singly-even self-dual code with minimal shadow, hence $d=4m+4$ and
$\wt(S)=r$ if $r=1,2,3$ and $\wt(S)=4$ if $r=0$. Obviously in this case
$a_0=1$, $a_1=a_2=\cdots=a_{2m+1}=0$. According to Theorem 1, we have
$b_0=1$ if $r>0$ and $m\ge 1$, and $b_0=0, b_1=1$ if $r=0$ and $m\ge 2$.

Moreover, if $r>0$ and $m\ge 1$ then $b_1=b_2=\cdots=b_{m-1}=0$.
Otherwise $S$ would contain a vector $v$ of weight less than or
equal to $4m-4+r$, and if $u\in S$ is a vector of weight $r$ then
$u+v\in C$ with $\wt(u+v)\le 4m+2r-4\le 4m+2$, a contradiction to
the minimum distance of $C$. Similarly, if $r=0$ and $m\ge 2$ then
$b_2=\cdots=b_{m-1}=0$.

\begin{remark} \label{remark1} \rm
For  extremal self-dual codes of length $24m+8l+2$ we furthermore have  $b_m=0$. Otherwise $S$ would contain a vector $v$ of weight $4m+1$, and if $u\in S$ is the vector of weight 1 which exists since $\wt(S)=1$, then $u+v\in C$ with $\wt(u+v)\le 4m+2$ contradicting the minimum distance of $C$.
\end{remark}

If $m\ge 2$ we have by (\ref{eq1})
\begin{equation} \label{eq2}
c_{2m+1}=\alpha_{2m+1,0}=\beta_{2m+1,\epsilon}+\sum_{j=m}^{m+l-1}\beta_{2m+1,j}b_j,
\end{equation}
where $\epsilon =1$ for $r=0$ and $\epsilon =0$ otherwise, since
$3m+l-2m-1 = m+l-1$. To evaluate this equation, which turns out to
be crucial in the following, we need to consider the coefficients
$\alpha_{i0}$ in details. In order to do this we denote by
$\alpha_{i}(n)$ the coefficient $\alpha_{i0}$ if $n$ is the length
of the code. According to \cite{Rains} we have
\begin{equation} \label{alpha}
\alpha_i(n)=\alpha_{i0}= -\frac{n}{2i}[\textrm{coeff. of}
\ y^{i-1} \ \textrm{in} \ (1+y)^{-n/2-1+4i}(1-y)^{-2i}].
\end{equation}
Let $t=4l+r$ and $n=24m+8l+2r=24m+2t$. Then

$$\alpha_{2m+1}(n)=-\frac{12m+t}{2m+1}
[\mbox{coeff.} \ \mbox{of} \ y^{2m} \ \mbox{in} \
(1+y)^{-12m-t-1+8m+4}(1-y)^{-4m-2}]$$
$$=-\frac{12m+t}{2m+1}
[\mbox{coeff.} \ \mbox{of} \ y^{2m} \ \mbox{in} \
(1+y)^{-4m-t+3}(1-y)^{-4m-2}]$$
For $t>5$ we obtain
$$\alpha_{2m+1}(n)=-\frac{12m+t}{2m+1}
[\mbox{coeff.} \ \mbox{of} \ y^{2m} \ \mbox{in} \
(1-y^2)^{-4m-t+3}(1-y)^{t-5}],$$
and if $t\le 5$ then $$\alpha_{2m+1}(n)=-\frac{12m+t}{2m+1}
[\mbox{coeff.} \ \mbox{of} \ y^{2m} \ \mbox{in} \
(1-y^2)^{-4m-2}(1+y)^{5-t}].$$
Since
$$(1-y^2)^{-a}=\sum_{0\le j} {-a\choose
j}(-1)^jy^{2j}=\sum_{0\le j} {a+j-1\choose j}y^{2j} \ \ \mbox{for} \ a>0,$$ it follows  in case $t\le 5$ that
\begin{align*}
\alpha_{2m+1}(n)&=-\frac{12m+t}{2m+1}
[\mbox{coeff.} \ \mbox{of} \ y^{2m} \ \mbox{in} \
(1+y)^{5-t}\sum_{j=0}^{m} {4m+j+1\choose j}y^{2j}]\\
&=-\frac{12m+t}{2m+1}\sum_{s=0}^{[\frac{5-t}{2}]}{5-t\choose 2s}{5m+1-s\choose m-s},
 \end{align*}
and in case $t>5$ that
\begin{align*}
\alpha_{2m+1}(n)&=-\frac{12m+t}{2m+1}
[\mbox{coeff.} \ \mbox{of} \ y^{2m} \ \mbox{in} \
(1-y)^{t-5}\sum_{j=0}^{m} {4m+t+j-4\choose j}y^{2j}]\\
&=-\frac{12m+t}{2m+1}\sum_{s=0}^{[\frac{t-5}{2}]}{t-5\choose 2s}{5m+t-4-s\choose m-s}.
\end{align*}
For the different lengths $n$ the values of $\alpha_{2m+1}(n)$ are listed in Table \ref{Table:a2m+1}.

\begin{table}[htb]
\centering \caption{The values $\alpha_{2m+1}(n)$ for  extremal self-dual codes} \vspace*{0.2in}
\label{Table:a2m+1}
{\footnotesize
\begin{tabular}{|c|c|c|c|}
\noalign{\hrule height1pt} 
$n$&$24m+2$&$24m+10$&$24m+18$\\
\hline
$\alpha_{2m+1}$&$-\ds\frac{(12m+1)(56m+4)}{(2m+1)(m-1)}{5m-1\choose m-2}$&$-\ds\frac{12m+5}{2m+1}{5m+1\choose m}$&$-\ds\frac{12(7m+5)(4m+3)}{m(m-1)}{5m+3\choose m-2}$\\
\hline\hline
$n$&$24m+4$&$24m+12$&$24m+20$\\
\hline
$\alpha_{2m+1}$&$-\ds\frac{2(6m+1)(8m+1)}{m(2m+1)}{5m\choose m-1}$&$-6\ds{5m+2\choose m}$& $-\ds\frac{20(6m+5)(4m+3)}{m(m-1)}{5m+4\choose m-2}$\\
\hline\hline
$n$&$24m+6$&$24m+14$&$24m+22$\\
\hline
$\alpha_{2m+1}$&$-\ds\frac{3(4m+1)(6m+1)}{m(2m+1)}{5m\choose m-1}$&$-\ds\frac{3(12m+7)}{m}{5m+2\choose m-1}$&$-\ds\frac{6(12m+11)(6m+5)(8m+7)}{m(m-1)(m-2)}{5m+4\choose m-3}$\\
\hline\hline
$n$&$24m+8$&$24m+16$&\\
\hline
$\alpha_{2m+1}$&$-\ds\frac{4(3m+1)}{2m+1}{5m+1\choose m}$&$-\ds\frac{16(3m+2)}{m}{5m+3\choose m-1}$&\\
\noalign{\hrule height1pt}
\end{tabular}
}
\end{table}
\bigskip

To evaluate equation (\ref{eq2}) we also need $\beta_{ij}$ which are known due to
\cite{Rains}. Here we have
\begin{equation} \label{beta}
\beta_{ij}=(-1)^i2^{-n/2+6i}\frac{k-j}{i}{k+i-j-1\choose k-i-j}, \end{equation}
where $k=\lfloor n/8\rfloor=3m+l$. In particular,
$$\beta_{2m+1,j}=-2^{6-t}\frac{3m+l-j}{2m+1}{5m+l-j\choose m+l-1-j} \quad \mbox{and} \quad \beta_{2m+1,m+l-1}=-2^{6-t}.$$

\medskip
Now we are prepared to prove:

\begin{thm} \label{non-existence}
Extremal self-dual codes of lengths $n=24m+2$, $24m+4$, $24m+6$, $24m+10$ and $24m+22$ with minimal shadow do not exist.
\end{thm}

\begin{proof}
According to \cite{Rains} any extremal self-dual code of length $24m+22$ has minimum distance $4m+6$ and the minimum weight of its shadow is $4m+7$.
Thus the shadow is not minimal since a minimal shadow must have minimum weight $3$.
 (There is a misprint in \cite{Rains} where it is stated that the minimum weight of the shadow is $4m+6$. But actually the weights in this shadow
 are of type $4j+3$).

In the other four cases we have
\begin{equation} \label{eq3}
c_{2m+1}=\alpha_{2m+1,0}=\beta_{2m+1,0}
\end{equation}
 by (\ref{eq2}).
In case $n=24m +10$ we use the fact that $b_m=0$, according to Remark \ref{remark1}.

 Simplifying  equation (\ref{eq3}) according to Table \ref{Table:a2m+1}  we obtain
\begin{align*}
48m^2+26m+1=0,& \ \ \ \ \mbox{if} \ n=24m+2\\
24m^2+14m+1=0,& \ \ \ \ \mbox{if} \ n=24m+4\\
48m^2+30m+3=0,& \ \ \ \ \mbox{if} \ n=24m+6\\
6m+3=0,& \ \ \ \ \mbox{if} \ n=24m+10.
\end{align*}

Since all these equations have no solutions $m\geq 0$ extremal self-dual codes with minimal shadow  do not exist for $ n \equiv 2,4,6,10 \bmod 24$.
\end{proof}

\begin{remark}\rm So far
no extremal self-dual codes of length $24m+2t$ are known  for $t=1,2,3,5$. According to \cite{nonexistence}  extremal self-dual codes of length $24m+2r$
do not exist for $r=1,2,3$ and $m=1,2,\dots,6,8,\dots,12,16,\dots,22$. Thus if there is (for instance) a self-dual $[170, 85, 32]$ code it will not have minimal shadow, by Theorem \ref{non-existence}.
\end{remark}

The next result is a crucial observation in order to prove explicit bounds for the existence of extremal singly-even self-dual codes.

\begin{thm}\label{unique}
Extremal singly-even self-dual codes with minimal shadow of lengths $n=24m+8$, $24m+12$, $24m+14$ and $24m+18$  have uniquely
determined weight enumerators.
\end{thm}

\begin{proof}
For $m=0$ and $m=1$ see Remark 3 and the examples at the end of the paper. Now let $m\ge 2$.

In case $n=24m+12$ or $n=24m+14$
we have $$c_i=\alpha_{i0}=\beta_{i0}+\sum_{j=m}^{3m+1-i}\beta_{ij}b_j \quad \mbox{for} \  i\le 2m+1 \quad \mbox{and} $$
$$ c_i=\alpha_{i0}+\sum_{j=2m+2}^{i}\alpha_{ij}a_j=\beta_{i0} \quad \mbox{for} \  i>2m+1.$$
Therefore $c_i=\alpha_{i0}$ for $i=0,1,\dots,2m+1$
and $c_i=\beta_{i0}$ for $i=2m+2,\dots,3m+1$.

In the case $n=24m+8$ we have $b_0=0$, $b_1=1$ and $b_2=\cdots=b_{m-1}=0$. Hence $c_i=\alpha_{i0}$ for $i=0,1,\dots,2m+1$
and $c_i=\beta_{i1}$ for $i=2m+2,\dots,3m+1$.

Similarly, if $n=24m+18$ we obtain $c_i=\alpha_{i0}$ for
$i=0,1,\dots,2m+1$ and $c_i=\beta_{i0}$ for $i=2m+2,\dots,3m+2$.
In both cases the weight enumerator can be computed as above.

By (\ref{alpha}) and (\ref{beta}),
 the values of
$c_i$ can be calculated and they  depend only on the length $n$.
Thus the weight enumerators are unique in all cases.
\end{proof}

In \cite{Zhang}, Zhang obtained upper bounds for the lengths of the extremal binary doubly-even codes. He proved that extremal doubly-even codes of length $n=24m+8l$ do not exist if $m\ge 154$ (for $l=0$), $m\ge 159$ (for $l=1$) and $m\ge 164$ (for $l=2$). For  extremal singly-even codes there is also a bound due to Rains \cite{Rains-asymptotic}. Unfortunately, he only states the existence of a bound. In the next corollary we give explicit bounds for
extremal singly-even self-dual codes with minimal shadow for lengths congruent 8, 12, 14 and 18 $\bmod \, 24$.

In  the proof we need the value of $c_{2m}=\alpha_{2m,0}$. According to \cite{Rains} we have
\begin{align*}
\alpha_{2m}(n)&=-\frac{24m+2t}{4m}[\textrm{coeff. of}
\ y^{2m-1} \ \textrm{in} \ (1+y)^{-4m-t-1}(1-y)^{-4m}]\\
&=-\frac{12m+t}{2m}[\textrm{coeff. of}
\ y^{2m-1} \ \textrm{in} \ (1-y)^{t+1}(1-y^2)^{-4m-t-1}]\\
&=-\frac{12m+t}{2m}
[\mbox{coeff.} \ \mbox{of} \ y^{2m-1} \ \mbox{in} \
(1-y)^{t+1}\sum_{j=0}^{m} {4m+t+j\choose j}y^{2j}]\\
&=\frac{12m+t}{2m}\sum_{s=1}^{[\frac{t+2}{2}]}{t+1\choose 2s-1}{5m+t-s\choose m-s}
\end{align*}
where $t=4l +r$ and $n=24m+8l +2r = 24m + 2t$.
The values for $\alpha_{2m}(n)$ are listed in Table \ref{Table:a2m}.

\begin{table}[htb]
\centering \caption{The values $\alpha_{2m}(n)$ for an extremal self-dual $[n=24m+2t,\frac{n}{2},4m+4]$ code} \vspace*{0.2in}
\label{Table:a2m}
{\footnotesize
\begin{tabular}{|c|c|}
\noalign{\hrule height1pt} 
$n$&$\alpha_{2m}(n)$\\
\hline
$24m+8$& $\ds\frac{8(4m+1)(11m+3)(3m+1)}{m(m-1)(m-2)}{5m+1\choose m-3}$\\
\hline
$24m+12$& $\ds\frac{24(116m^2+79m+15)(1+2m)^2}{m(m-1)(m-2)(m-3)}{5m+2\choose m-4}$\\
\hline
$24m+14$& $\ds\frac{24(1+2m)(12m+7)(28m^2+22m+5)}{m(m-1)(m-2)(m-3)}{5m+3\choose m-4}$\\
\hline
$24m+16$& $\ds\frac{16(3m+2)(2m+1)(1216m^3+1956m^2+1073m+210)}{m(m-1)(m-2)(m-3)(m-4)}{5m+3\choose m-5}$\\
\hline
$24m+18$& $\ds\frac{120(2m+1)(4m+3)(176m^3+308m^2+189m+42)}{m(m-1)(m-2)(m-3)(m-4)}{5m+4\choose m-5}$\\
\hline
$24m+20$& $\ds\frac{16(6m+5)(2m+1)(4m+3)(1592m^3+3280m^2+2363m+630)}{m(m-1)(m-2)(m-3)(m-4)(m-5)}{5m+4\choose m-6}$\\
\noalign{\hrule height1pt}
\end{tabular}
}
\end{table}

\bigskip
Furthermore, $\beta_{2m,j}=2^{-t}\ds\frac{3m+l-j}{2m}{5m+l-1-j\choose m+l-j}$. Hence $\beta_{2m,m+l}=2^{-t}$ and $\beta_{2m,m+l-1}=2^{1-t}(2m+1)$.

\begin{cor} \label{bound}
There are no extremal singly-even self-dual codes of length $n$ with minimal shadow if
\begin{itemize}
\item[(i)] $n=24m+8$ and $m\ge 53$,
\item[(ii)] $n=24m+12$ and $m\ge 142$,
\item[(iii)] $n=24m+14$ and $m\ge 146$,
\item[(iv)] $n=24m+18$ and $m\ge 157$.
\end{itemize}
\end{cor}

\begin{proof}
Using the equation
 $$c_i=\alpha_{i0}=\beta_{i\epsilon }+\sum_{j=m}^{3m+l-i}\beta_{ij}b_j \quad \mbox{for} \  i\le 2m+1, $$
 where $\epsilon =1$ if $n=24m+8$ and $\epsilon =0$ in the other cases,
we see that
$$b_{m+l-1}=-2^{t-6}(\alpha_{2m+1,0}-\beta_{2m+1,\epsilon}).$$
The values of $b_m$ for $n=24m+8$, $24m+12$ and $24m+14$ are given in Table \ref{Table:bm}.

\begin{table}[htb]
\centering \caption{The parameter $b_{m}$ for extremal self-dual codes of length $n$} \vspace*{0.2in}
\label{Table:bm}
{\footnotesize
\begin{tabular}{|c|c|c|c|}
\noalign{\hrule height1pt} 
$n$&$24m+8$&$24m+12$&$24m+14$\\
\hline
$b_m$&$\ds\frac{6m+1}{m}{5m\choose m-1}$&$\ds\frac{12m+5}{2m+1}{5m+1\choose m}$& $\ds\frac{168m^2+164m+39}{(2m+1)(4m+3)}{5m+1\choose m}$\\
\noalign{\hrule height1pt}
\end{tabular}
}
\end{table}
\noindent
If $n=24m+18$ we have $$b_m=0 \quad \mbox{and} \quad
b_{m+1}=\frac{(24m+17)(17m+10)}{(2m+1)(4m+5)}{5m+2\choose m+1}.$$
In the first three  cases we compute
$$ b_{m+1} = \frac{\alpha_{2m,0} - \beta_{2m,\epsilon} - \beta_{2m,m}b_m}{\beta_{2m,m+1}}.$$

If $n=24m+8$ we obtain
$$b_{m+1}=\frac{16(6m+1)(-4m^3+209m^2+141m+24)}{5m(m+1)(4m+3)}{5m+1\choose m-1}$$
In case $m\ge 53$ the polynomial $-4m^3+209m^2+141m+24$ takes negative values, hence $b_{m+1}<0$, a contradiction.

For $24m+12$ we have
$$b_{m+1}= \frac{2(12m+5)(-32m^4+4496m^3+4242m^2+1257m+117)}{(5m+1)(4m+3)(4m+5)(2m+3)}
{5m+2\choose m+1}$$
If $m\ge 142$ the polynomial $-32m^4+4496m^3+4242m^2+1257m+117$ takes negative values, hence $b_{m+1}<0$, a contradiction.

For $24m+14$ the calculations lead to
$$b_{m+1}= \frac{2(-5376m^6+772352m^5+1663728m^4+1386448m^3+557970m^2+107643m+7875)}{(4m+3)(4m+5)(2m+3)(4m+7)(5m+1)}{5m+2\choose m+1}$$
which is negative if $m\ge 146$.

In the last case we have to compute
$$ b_{m+2} = \frac{\alpha_{2m,0} - \beta_{2m,0} - \beta_{2m,m+1}b_{m+1}}{\beta_{2m,m+2}}.$$
The computations yield
$$b_{m+2}=\frac{ 2(24m+17)(-544m^5+83696m^4+184210m^3+149089m^2+52809m+6930)}{(4m+5)(2m+3)(4m+7)(4m+9)(5m+2)} {5m+3\choose m+2}$$
which is negative for $m\ge 157$.
\end{proof}

\begin{prop} \label{de}
If there are no extremal doubly-even self-dual codes of length $n=24m+8$ or $24m+16$  then there are no extremal singly-even self-dual codes of length $n$ with minimal shadow.
\end{prop}

\begin{proof}
We shall prove the contraposition. Let $C$ be a singly-even
self-dual $[n=24m+8l,12m+4l,4m+4]$ code and suppose that the coset
$C_1$ contains the vector $u$ of weight 4. If $v\in C_3$ then
$u+v\in C_2$ and hence $\wt(u+v)\ge 4m+6$. It follows that
$$\wt(v)\ge 4m+6-4+2\wt(u*v)\ge 4m+4,$$ since $C_1$ is not orthogonal to $C_3$, which means that $u*v\equiv 1\pmod 2$ for $u\in C_1, v\in C_3$ (see \cite{BrualdiPless}).
Thus $\wt(C_3)\ge 4m+4$.
Therefore $C_0\cup C_3$ is an extremal doubly-even code with
parameters $[24m+8l,12m+4l,4m+4]$.
\end{proof}

\begin{cor} \label{24m+16} There are no extremal singly-even self-dual codes with minimal shadow of length $n=24m +16$ for $ m \geq 164$.
\end{cor}
\begin{proof} This follows immediately from the Zhang bound \cite{Zhang} for doubly-even codes in connection with Proposition \ref{de}.
\end{proof}

Summarizing the results in Theorem \ref{non-existence}, Corollary \ref{bound} and Corollary \ref{24m+16} we have proved either
the non-existence or an explicit bound for the length $n$ of an extremal singly-even self-dual code unless $n \equiv 20 \pmod {24}$.
To find an explicit bound for $n=24m +20$ seems to be difficult since the weight enumerator is not unique in this case.

\begin{remark} \rm Extremal singly-even self-dual codes of length $24m+8$ are constructed only for $m=1$, i.e. $n= 32$. There are exactly three inequivalent  singly-even self-dual $[32,16,8]$ codes. Yorgov proved that there are no extremal singly-even self-dual codes with minimal shadow of length $24m+8$ in the case  $m$ is even and $5m\choose m$ is odd \cite{Yorgov99}. 
\end{remark}

\noindent
{\bf Examples.} \rm Extremal singly-even self-dual codes  of lengths $24m+12$, $24m+14$ and $24m+18$: \\[1ex]
$m=0$: \  There are unique extremal singly-even codes of lengths 12, 14 and 16, and they have minimal shadows.
There are two inequivalent self-dual $[18,9,4]$ codes, but only one of them is a code with minimal shadow (see \cite{CS}).

\noindent
$m=1$: \  Extremal self-dual codes  of lengths 36, 38 and 42 with minimal shadow are constructed. Only for the length 36 there is a complete classification \cite{Huffman05}. There are 16 inequivalent self-dual $[36,18,8]$ codes with minimal shadow and their weight enumerator is $W=1 + 225y^8 + 2016y^{10} + 9555y^{12} +\cdots$ (see \cite{n=36}).

\noindent
$m=2$: \  There exists a doubly circulant code with parameters $[60,30,12]$ and shadow of minimum weight $2$, denoted by $D13$ in \cite{CS}. The first examples for extremal self-dual codes with minimal shadow of lengths 62 and 66 are constructed in \cite{aut9} and \cite{Tsai}, respectively.


\bigskip

Finally, we would like to mention that
similar to the case of extremal doubly-even self-dual codes
there is a large gap between the bounds for extremal singly-even self-dual codes and what we really can construct.


\end{document}